\documentclass[a4paper,12pt]{article}
\usepackage{amsmath,amssymb,amsthm,graphics,latexsym,amsfonts}
\usepackage{fancyhdr}
\usepackage{color}
\usepackage{graphics}
\usepackage{epstopdf}
\usepackage{amssymb}
\usepackage{cite}
\usepackage{hyperref}
\hypersetup{
    colorlinks=true,
    linkcolor=cyan,
    filecolor=cyan,
    urlcolor=cyan,
    citecolor=cyan,
}

\title{\Large Odd spanning trees of a graph\,\thanks{Research supported by Open Project of Key Laboratory in Xinjiang Uygur Autonomous Region of China (2023D04026) and by NSFC (No. 12061073).}}
\author{ {Jingyu Zheng, Baoyindureng Wu \footnote{Corresponding author.
Email: baoywu@163.com (B. Wu) }}\\
\small  College of Mathematics and System Sciences, Xinjiang
University \\ \small  Urumqi, Xinjiang 830046, P.R.China \\}
\date{}

\newtheorem{theorem}{Theorem}[section]
\newtheorem{lemma}[theorem]{Lemma}
\newtheorem{corollary}[theorem]{Corollary}

\newtheorem{proposition}[theorem]{Proposition}

\usepackage{indentfirst}

\begin{document}

\maketitle {\small \noindent{\bfseries Abstract}
A graph $G=(V,E)$ is said to be odd (or even, resp.) if $d_G(v)$ is odd (or even, resp.) for any $v\in V$. Trivially, the order of an odd graph must be even. In this paper, we show that every 4-edge connected graph of even order has a connected odd factor. 
A spanning tree $T$ of $G$ is called a homeomorphically irreducible spanning tree (HIST by simply) if $T$ contains no vertex of degree two. Trivially, an odd spanning tree must be a HIST.
In 1990, Albertson, Berman, Hutchinson, and Thomassen showed that every connected graph of order $n$ with $\delta(G)\geq \min\{\frac n 2, 4\sqrt{2n}\}$ contains a HIST. 

We show that every complete bipartite graph with both parts being even has no odd spanning tree, thereby for any even integer $n$ divisible by 4, there exists a graph of order $n$ with the minimum degree $\frac n 2$ having no odd spanning tree. Furthermore, we show that every graph of order $n$ with $\delta(G)\geq \frac n 2 +1$ has an odd spanning tree. We also characterize all split graphs having an odd spanning tree. As an application, for any graph $G$ with diameter at least 4, $\overline{G}$ has a spanning odd double star. Finally, we also give a necessary and sufficient condition for a triangle-free graph $G$ whose complement contains an odd spanning tree. A number of related open problems are proposed.\\
{\bfseries Keywords:} Complements; Odd graphs; Spanning trees; Split graphs; Triangle-free graphs\\
{\bfseries Mathematics Subject Classification:} 05C05; 05C07

\section {\large Introduction}

All graphs considered here are simple and finite. For graph-theoretic notation not explained in this paper, we refer to \cite{Bondy}.
Let $G=(V(G), E(G))$ be a graph. The order and size of $G$ are often denoted by $v(G)$ and $e(G)$.
The set of neighbours of a vertex $v$ in a graph $G$ is denoted by $N_{G}(v)$. The degree of a vertex $v$, denoted by $d_G(v)$, is the number of edges incident with $v$ in $G$. Since $G$ is simple, $d_G(v)=|N_G(v)|$. The maximum and minimum degrees of $G$ are denoted by $\Delta(G)$ and $\delta(G)$, respectively. We say $G$ is {\it even} (or {\it odd}, resp.) if all vertices have degree even (or odd, resp.).
A spanning subgraph of $G$ is often called a {\it factor} of $G$. In particular, a spanning tree where the vertices are odd degrees is called an {\it odd spanning tree}.
For a set $S\subseteq V(G)$, $G-S$ denotes the graph obtained from $G$ by removing vertices of $S$ and all edges that are incident to a vertex of $S$. The subgraph $G-(V(G)\setminus S)$ is said to be the {\it induced subgraph} of $G$ induced by $S$, and
is denoted by $G[S]$. For $A, B\subseteq V(G)$, $E_G[A, B]$ denotes the set of edges with one end in $A$ and the other end in $B$, and $e_G(A, B)=|E_G[A,B]|$.
The distance of two vertices $u$ and $v$, denoted by $d_G(u,v)$, is the length of shortest path joining $u$ and $v$ in $G$. The diameter of $G$, denoted by $diam(G)$, is $\max\{d_G(u,v): u, v\in V(G)\}$.

For a graph $G$, its complement $\overline{G}$ is the graph with $V(G)$, in which two vertices are adjacent if and only if they are not adjacent in $G$.
A bipartite graph with bipartition $(X, Y)$ is denoted by $G[X,Y]$. In particular, $K_{m,n}$ denotes the complete bipartite graph with $m$ and $n$ vertices in its two parts.
A triangle-free graph is one which contains no triangles.

Gallai (see \cite{Lovasz1979}, Section 5, Problem 17) proved that the vertices of any graph can be partitioned into two sets, each of which induces a subgraph with all degrees even; the vertices also can be partitioned into two sets so that one set induces a subgraph with all degrees even and the other induces a subgraph with all degrees odd. We refer to \cite{Berman1997, Berman, Caro19941, Caro19942, Ferber2021, Gutin2016, Hou2018, Radcliff1995, Rao2022, Scott1992, Wang2023, WangWu2024} for the large odd induced subgraph of a graph. Scott \cite{Scott2001} showed that the following theorem.

\begin{theorem} [\cite{Scott2001}]\label{Theorem1.1}
Every connected graph of even order has a vertex partition into sets inducing subgraphs with all degrees odd.
\end{theorem}

The above theorem indicates that every connected graph of order even has an odd factor. It is natural to ask that when a graph has a connected odd factor or particularly, has an odd spanning tree.
These are exactly our main subjects here. We show that every 4-edge connected graph has a connected odd factor.

Recall that a spanning tree $T$ of $G$ is called a homeomorphically irreducible spanning tree (HIST by simply) if $T$ contains no vertex of degree two. Trivially, an odd spanning tree must be a HIST. The existence of a HIST in a graph is extensively studied in literature \cite{Albertson1990, Chen2012, Chen2013, Hoffmann-Ostenhof2018, Diemunsch2015, Furuya2020, Ito2020, Ito2022, Shan2023, Zhai2019}.
In particular, it was shown that every connected graph $G$ of order $n$ has a HIST if one of the following conditions is satisfied:

(1) $\delta(G)\geq 4\sqrt{2n}$, or $\delta(G)\geq \frac n 2$ and $n\geq 4$, by Albertson, Berman, Hutchinson, Thomassen \cite{Albertson1990};

(2) $n\geq 8$ and $d(u)+d(v)\geq n-1$ for any two nonadjacent vertices $u, v$, by Ito and Tsuchiya \cite{Ito2022};

(3) graphs with diameter two except a well-defined families of graphs, by Shan and Tsuchiya \cite{Shan2023};

(4) connected and locally connected graphs of order $n\geq 3$, by Chen, Ren and Shan \cite{Chen2012}.

Motivated by the above results, we show that every graph of even order $n$ with $\delta(G)\geq \frac n 2 +1$ has an odd spanning tree. This is best possible, in sense that there is a graph of order $n$ divisible by 4 with $\delta(G)=\frac n 2$ has no odd spanning tree. We also characterize all split graphs having an odd spanning tree. As an application, for any graph $G$ with diameter at least 4, $\overline{G}$ has a spanning odd double star. Finally, we also give a necessary and sufficient condition for a triangle-free graph $G$ whose complement contains odd spanning tree. In addition, some interesting by-products are also detected: a connected graph $G$ is bipartite if and only if any two spanning trees of $G$ have the same bipartition; the complement of a triangle-free graph $G$ is connected if and only if $G$ is not a complete bipartite graph.

\section{\large Connected odd factors and odd spanning trees}

First recall a well-known theorem for a graph having $k$ edge-disjoint spanning trees, due to Nash-Williams and Tutte.

\begin{theorem}[\cite{Nash-Williams1961, Tutte1961}]\label{Theorem2.1}
A graph $G$ has $k$ edge-disjoint spanning trees if and only if$$|E_G(P)|\geq k(|P|-1)$$for every partition $P$ of $V(G)$, where $E_G(P)$  denotes the set of edges of $G$  joining different parts of $P$.
\end{theorem}

For two subgraphs $F$ and $H$ of a graph $G$, $F\bigtriangleup H$ is the spanning subgraph of $G$ with edge set $E(F)\bigtriangleup E(H)$.

\begin{theorem}\label{Theorem2.2}
Every 4-edge-connected graph $G$ of even order has a connected odd factor.
\end{theorem}


\begin{proof} Let $G$ be 4-edge-connected graph of even order.
By \hyperref[Theorem2.1]{Theorem 2.1}, $G$ has two edge-disjoint spanning trees $T_1$ and $T_2$. If $T_1$ or $T_2$ is odd, we are done. So, let $\{x_1, x_2, \ldots, x_k, y_1, y_2, \ldots, y_k\}$ be the set of all vertices of even degree in $T_1$, where $k$ is a positive integer.
One can find a unique path $P_i$ in $T_2$ joining $x_i$ and $y_i$ for each $i$. Clearly,
$T_1'=T_1\bigtriangleup P_1\bigtriangleup P_2\bigtriangleup \ldots \bigtriangleup P_k$ is a connected odd factor of $G$, as we desired.
\end{proof}

Recall that a graph is bipartite if and only if it contains no odd cycle. Next, we give another interesting criterion for a graph being bipartite, which will be used in the proof of the next proposition.

\begin{lemma}\label{Lemma2.3}
A connected graph $G$ is bipartite if and only if any two spanning trees of $G$ have the same bipartition.
\end{lemma}

\begin{proof}

First assume that $G=G[X,Y]$ is a connected bipartite graph. The necessity follows from the fact that all connected spanning subgraph of $G$ have the same bipartition as $G$ does.

To prove the sufficiency, let $T_1$ and $T_2$ be two spanning trees of $G$ have different bipartitions.
It follows that there exist two vertices $u$ and $v$ which lie in the same part of $T_1$ and lie in the distinct parts of $T_2$ (if necessary, we may change the role of $T_1$ and $T_2$). Let $P_i$ be the unique path joining $u$ and $v$ in $T_i$ for each $i\in \{1, 2\}$. Clearly, $P_1\cup P_2$ is a closed walk with length being odd. So, $P_1\cup P_2$ contains an odd cycle and thus $G$ has an odd cycle, contradicting the assumption that $G$ is bipartite.
\end{proof}

\begin{proposition}\label{Proposition2.4}
Let $G=G[X,Y]$ be a bipartite graph. If both $|X|$ and $|Y|$ are even, then $G$ has no odd spanning tree. In particular, $K_{2s, 2t}$ has no odd spanning tree for any two integers $s, t\geq 1$.
\end{proposition}

\begin{proof}
Suppose that $G$ has an odd spanning tree $T$. By \hyperref[Lemma2.3]{Lemma 2.3}, $T$ and $G$ have the same bipartition $(X, Y)$. In one hand, $\sum_{x\in X} d_T(x)=|X|+|Y|-1\equiv 1 ~ (\hspace{-2.5mm}\mod 2)$. On the other hand, since $d_T(x)$ is odd for each $x\in X$ and $|X|$ is even, we have $\sum_{x\in X} d_T(x)\equiv 0 ~ (\hspace{-2.5mm}\mod 2)$. This is a contradiction.
\end{proof}

\begin{proposition}\label{Proposition2.5}
Let two graphs $G_1$ and $G_2$ be vertex-disjoint.
Let $G$ be the graph obtained by joining a vertex of $G_1$ to a vertex of $G_2$. If both $G_1$ and $G_2$ have order even, then $G$ has no odd spanning tree. In particular, the result holds for $G_1=K_s$ and $G_2=K_t$, where both $s$ and $t$ are even.
\end{proposition}

\begin{proof} Let $T$ be an odd spanning tree of $G$. Clearly, $T$ contains the edge $e$ joining $G_1$ and $G_2$. Without loss of generality, the order $s$ of $G_1$ is even.
Let $T_1$ be the component of $T-e$ contained in $G_1$. However, $s-1$ is odd, but $T_1$ has $s-1$ vertices of degree odd, a contradiction.
\end{proof}

%
%

By \hyperref[Proposition2.4]{Proposition 2.4}, $K_{\frac n 2 , \frac n 2}$ has no odd spanning tree if $n$ is divisible by 4. This indicates that $\delta(G)\geq \frac n 2 +1$ in the following theorem is best possible for any $n$ divisible by 4. In addition, by \hyperref[Proposition2.5]{Proposition 2.5}, for any positive even number $n$, there exists a connected graph $G_n$ of order $n$ with $\delta(G_n)\geq \frac n 2-1$ having no odd spanning tree.

\vspace{2mm} For a vertex subset $X$ of a graph $G$ and $y\in V(G)$, let $N_X(y)=N_G(y)\cap X$.

\begin{theorem}\label{Theorem2.6}
 Let $n$ be a positive even number. If $G$ is a connected graph of order $n$ with $\delta(G)\geq \frac{n}{2}+1$, then $G$ has an odd spanning tree.
\end{theorem}
\begin{proof}
We divide into two cases by the parity of the degrees of the vertices.

\vspace{2mm}\noindent{\bf Case 1.} $G$ is not even.

\vspace{2mm} Take a vertex $v$ with $d_G(v)$ being odd. Let $T$ be an odd subtree of $G$ with $d_T(v)=d_G(v)$ such that $|V(T)|$ is as large as possible.
Since $d_T(v)=d_G(v)$,
\begin{equation}
|V(T)|\geq d_G(v)+1\geq \delta(G)+1\geq \frac n 2 +2.
\end{equation}

We claim that $T$ must be a spanning tree of $G$.
Suppose it is not, and let $G_2=G-V(T)$. By the maximality of $T$,
\begin{equation}
|N_G(x)\setminus V(T)|\leq 1\ \text{ for any}\ x\in V(T).
\end{equation}
In addition, since $|V(T)|\geq \frac{n}{2}+2$, we have $|V(G_2)|\leq \frac{n}{2}-2$, and thus $\Delta(G_2)\leq \frac{n}{2}-3$. Hence, for any $y\in V(G_2)$,
\begin{equation}
d_T(y)=d_G(y)-d_{G_2}(y)\geq \delta(G)-\Delta(G_2)\geq (\frac n 2 +1)-(\frac{n}{2}-3)=4.
\end{equation}
Combining (2) and (3), one has
\begin{equation}
4|V(G_2)|\leq e_G(V(T), V(G_2))\leq |V(T)|.
\end{equation}
Associating (4) with $|V(G_2)|+|V(T)|=n$, it gives $|V(G_2)|\leq \frac{n}{5}$. By taking this into the following, we obtain for any $y\in V(G_2)$,
\begin{equation}
d_T(y)=d_G(y)-d_{G_2}(y)\geq \delta(G)-\Delta(G_2)\geq (\frac n 2 +1)-(\frac{n}{5}-1)=\frac{3n}{10}+2.
\end{equation}
Combining (2) and (5), we have
\begin{equation}
(\frac{3n}{10}+2)|V(G_2)|\leq e_G(V(T), V(G_2))\leq |V(T)|.
\end{equation}
Integrating (6) with $|V(G_2)|+|V(T)|=n$,  $|V(G_2)|\leq \frac{10n}{3n+30}<\frac{10}{3}$.
Since both $|V(T)|$ and $n$ are even, $|V(G_2)|$ is even. It follows that $|V(G_2)|=2$.
Let $V(G_2)=\{x, y\}$. Since $\delta(G)\geq \frac{n}{2}+1$,
$|N_G(x)\cap N_G(y)|\geq 2$. Take a vertex $z\in N_G(x)\cap N_G(y)$. Clearly,
$T+\{zx, zy\}$ is a larger odd subtree of $G$ than $T$, a contradiction.
This proves that $T$ is an odd spanning tree of $G$.

\vspace{2mm}\noindent{\bf Case 2.} $G$ is even.

\vspace{2mm}\noindent{\bf Subcase 2.1.} $\Delta(G)>\delta(G)$

\vspace{2mm} Take a vertex $v$ with $d_G(v)=\Delta(G)$ and $v'\in N_G(v)$ and let $H=G-vv'$.
Let $T$ be an odd tree of $H$ with $d_T(v)=d_G(v)-1$ such that $|V(T)|$ is as large as possible.
Since $d_T(v)=d_G(v)-1$,
\begin{equation}
|V(T)|\geq d_G(v)\geq \frac n 2 +2.
\end{equation}

We claim that $T$ is a spanning tree.
Suppose $T$ is not, and let $G_2=H-V(T)$. By the maximality of $T$,
\begin{equation}
|N_{H}(x)\setminus V(T)|\leq 1\ \text{ for any}\ x\in V(T).
\end{equation}

In addition, since $|V(T)|\geq \frac{n}{2}+2$, we have $|V(G_2)|\leq \frac{n}{2}-2$, and thus $\Delta(G_2)\leq \frac{n}{2}-3$. Hence, for any $y\in V(G_2)\setminus \{v'\}$,
\begin{equation}
d_T(y)=d_H(y)-d_{G_2}(y)\geq \delta(G)-\Delta(G_2)\geq (\frac n 2 +1)-(\frac{n}{2}-3)=4.
\end{equation}

\noindent{\bf Subcase 2.1.1.} $v'\in V(G_2)$

\vspace{2mm} Observe that
\begin{equation}
d_T(v')=d_H(v')-d_{G_2}(v')\geq \delta(H)-\Delta(G_2)\geq \frac n 2-(\frac{n}{2}-3)=3.
\end{equation}
Combining (8), (9), and (10), one has
\begin{equation}
4|V(G_2)|-1\leq e_G(V(T), V(G_2))\leq |V(T)|.
\end{equation}
Uniting (11) with $|V(G_2)|+|V(T)|=n$, it gives $|V(G_2)|\leq \frac{n+1}{5}$. By taking this into the following, we obtain for any $y\in V(G_2)\setminus \{v'\}$,
\begin{equation}
d_T(y)=d_H(y)-d_{G_2}(y)\geq \delta(G)-\Delta(G_2)\geq (\frac n 2 +1)-(\frac{n+1}{5}-1)=\frac{3n-2}{10}+2.
\end{equation}
\begin{equation}
d_T(v')\geq \frac{3n-2}{10}+1
\end{equation}
Combining (8), (12), and (13), we have
\begin{equation}
(\frac{3n-2}{10}+2)|V(G_2)|-1\leq e_G(V(T), V(G_2))\leq |V(T)|.
\end{equation}
Joining (14) with $|V(G_2)|+|V(T)|=n$,  $|V(G_2)|\leq \frac{10(n+1)}{3n+28}<\frac{10}{3}$.
Since both $|V(T)|$ and $n$ are even, $|V(G_2)|$ is even. It follows that $|V(G_2)|=2$.
Let $V(G_2)=\{x, v'\}$. Since $\delta(G)\geq \frac{n}{2}+1$,
$|N_G(x)\cap N_G(v')|\geq 2$. Take a vertex $z\in N_G(x)\cap N_G(v')$ distinct from $v$. Clearly,
$T+\{zx, zv'\}$ is a larger odd subtree of $G$ than $T$, a contradiction.
This proves that $T$ is an odd spanning tree of $G$.

\vspace{2mm}\noindent{\bf Subcase 2.1.2.} $v'\in V(T)$

\vspace{2mm} By the assumption, we have $N_G[v]\subseteq V(T)$ and therefore, $|V(T)|\geq \frac{n}{2}+3$.  Hence $|V(G_2)|\leq \frac{n}{2}-3$ and moreover, $\Delta(G_2)\leq \frac{n}{2}-4$. So, for any $y\in V(G_2)$,
\begin{equation}
d_T(y)=d_H(y)-d_{G_2}(y)\geq \delta(G)-\Delta(G_2)\geq (\frac n 2 +1)-(\frac{n}{2}-4)=5.
\end{equation}
Combining (8) and (15), one has
\begin{equation}
5|V(G_2)|\leq e_G(V(T), V(G_2))\leq |V(T)|.
\end{equation}
Merging (16) with $|V(G_2)|+|V(T)|=n$, it gives $|V(G_2)|\leq \frac{n}{6}$. By taking this into the following, we obtain for any $y\in V(G_2)$,
\begin{equation}
d_T(y)=d_H(y)-d_{G_2}(y)\geq \delta(G)-\Delta(G_2)\geq (\frac n 2 +1)-(\frac{n}{6}-1)=\frac{n}{3}+2.
\end{equation}
Combining (8) and (17), we have
\begin{equation}
(\frac{n}{3}+2)|V(G_2)|\leq e_G(V(T), V(G_2))\leq |V(T)|.
\end{equation}
Integrating (18) with $|V(G_2)|+|V(T)|=n$,  $|V(G_2)|\leq \frac{3n}{n+9}<3$.
Since both $v(T)$ and $n$ are even, $v(G_2)$ is even. It follows that $|V(G_2)|=2$.
Let $V(G_2)=\{x, y\}$. Since $\delta(G)\geq \frac{n}{2}+1$,
$|N_G(x)\cap N_G(y)|\geq 2$. Take a vertex $z\in N_G(x)\cap N_G(y)$. Clearly,
$T+\{zx, zy\}$ is a larger odd sub-tree of $G$ than $T$, a contradiction.
This proves that $T$ is an odd spanning tree of $G$.

\vspace{2mm}\noindent{\bf Subcase 2.2.} $\Delta(G)=\delta(G)$

\vspace{2mm}By the assumption, $G$ is $\delta(G)$-regular graph and $\delta(G)$ is even.
Let $vv'\in E(G)$ and $H=G-vv'$ as we did before.
Let $T$ be an odd tree of $H$ with $d_T(v)=d_G(v)-1$ such that $|V(T)|$ is as large as possible.
Since $d_T(v)=d_G(v)-1$,
\begin{equation}
|V(T)|\geq d_G(v)=\frac n 2 +1.
\end{equation}

We claim that $T$ is a spanning tree.
Suppose $T$ is not, and let $G_2=H-V(T)$. By the maximality of $T$,
\begin{equation}
|N_{H}(x)\setminus V(T)|\leq 1\ \text{ for any}\ x\in V(T).
\end{equation}

In addition, since $|V(T)|\geq \frac{n}{2}+1$, we have $|V(G_2)|\leq \frac{n}{2}-1$, and thus $\Delta(G_2)\leq \frac{n}{2}-2$. Hence, for any $y\in V(G_2)\setminus \{v'\}$,
\begin{equation}
d_T(y)=d_H(y)-d_{G_2}(y)\geq d_G(y)-\Delta(G_2)\geq (\frac n 2 +1)-(\frac{n}{2}-2)=3.
\end{equation}

\noindent{\bf Subcase 2.2.1.} $v'\in V(G_2)$

\vspace{2mm} Clearly,
\begin{equation}
d_T(v')=d_H(v')-d_{G_2}(v')\geq \frac n 2-(\frac{n}{2}-2)=2.
\end{equation}
Combining (20), (21), and (22) one has
\begin{equation}
3|V(G_2)|-1\leq e_G(V(T), V(G_2))\leq |V(T)|.
\end{equation}
Linking (23) with $|V(G_2)|+|V(T)|=n$, it gives $|V(G_2)|\leq \frac{n+1}{4}$. By taking this into the following, we obtain for any $y\in V(G_2)\setminus\{v'\}$,
\begin{equation}
d_T(y)=d_H(y)-d_{G_2}(y)\geq (\frac n 2 +1)-(\frac{n+1}{4}-1)=\frac{n+7}{4}.
\end{equation}
\begin{equation}
d_T(v')\geq \frac{n+3} 4.
\end{equation}
Combining (20), (24), and (25), we have
\begin{equation}
\frac{n+7}{4}|V(G_2)|-1\leq e_G(V(T), V(G_2))\leq |V(T)|.
\end{equation}
Joining (26) with $|V(G_2)|+|V(T)|=n$,  $|V(G_2)|\leq \frac{4(n+1)}{n+11}<4$.
Since both $|V(T)|$ and $n$ are even, $|V(G_2)|$ is even. It follows that $|V(G_2)|=2$.
Let $V(G_2)=\{x, v'\}$. Since $\delta(G)\geq \frac{n}{2}+1$,
$|N_G(x)\cap N_G(v')|\geq 2$. Take a vertex $z\in N_G(x)\cap N_G(v')$ distinct from $v$. Clearly,
$T+\{zx, zv'\}$ is a larger odd sub-tree of $G$ than $T$, a contradiction.
This proves that $T$ is an odd spanning tree of $G$.

\vspace{2mm}\noindent{\bf Subcase 2.2.2.} $v'\in V(T)$

\vspace{2mm}Since $|V(T)|\geq \frac{n}{2}+2$, we have $|V(G_2)|\leq \frac{n}{2}-2$, and thus $\Delta(G_2)\leq \frac{n}{2}-3$. Hence, for any $y\in V(G_2)$,
\begin{equation}
d_T(y)=d_H(y)-d_{G_2}(y)\geq (\frac n 2 +1)-(\frac{n}{2}-3)=4.
\end{equation}
Combining (20) and (27), one has
\begin{equation}
4|V(G_2)|\leq e_G(V(T), V(G_2))\leq |V(T)|.
\end{equation}
Integrating (28) with $|V(G_2)|+|V(T)|=n$, it gives $|V(G_2)|\leq \frac{n}{5}$. By taking this into the following, we obtain for any $y\in V(G_2)$,
\begin{equation}
d_T(y)=d_H(y)-d_{G_2}(y)\geq (\frac n 2 +1)-(\frac{n}{5}-1)=\frac{3n}{10}+2.
\end{equation}
Combining (20) and (29), we have
\begin{equation}
(\frac{3n}{10}+2)|V(G_2)|\leq e_G(V(T), V(G_2))\leq |V(T)|.
\end{equation}
Uniting (30) with $|V(G_2)|+|V(T)|=n$,  $|V(G_2)|\leq \frac{10n}{3n+30}<\frac{10}{3}$.
Since both $|V(T)|$ and $n$ are even, $|V(G_2)|$ is even. It follows that $|V(G_2)|=2$.
Let $V(G_2)=\{x, y\}$. Since $\delta(G)\geq \frac{n}{2}+1$,
$|N_G(x)\cap N_G(y)|\geq 2$. Take a vertex $z\in N_G(x)\cap N_G(y)$. Clearly,
$T+\{zx, zy\}$ is a larger odd sub-tree of $G$ than $T$, a contradiction.
This proves that $T$ is an odd spanning tree of $G$.
\end{proof}


\section{\large Split graphs}

A graph $G=(V,E)$ is said to be a {\it split graph} if there exists a partition $(X, Y)$ of $V$ such that $X$ is an independent set and $Y$ is a clique. We will denote such a graph by $G(X\cup Y, E)$.

\begin{theorem}\label{Theorem3.1}
Let $G=G(X\cup Y, E)$ be a connected split graph of an even order, where $X$ is an independent set and $Y$ is a clique. Then $G$ has no odd spanning tree if and only if $d_X(y)\equiv 1~(\hspace{-2.5mm}\mod 2)$ for any $y\in Y$ and $N_X(y)\cap N_X(y')=\emptyset$ for any two vertices $y$, $y'\in Y$.
\end{theorem}

\begin{proof} To show its sufficiency, assume that
$d_X(y)\equiv 1~(\hspace{-2.5mm}\mod 2)$ and $N_X(y)\cap N_X(y')=\emptyset$ for any two vertices $y$, $y'\in Y$. We show that $G$ has no odd spanning tree. Suppose $T$ is an odd spanning tree of $G$.
By the assumption, $T-X$ is a spanning tree of $G-X$. Let $y_i$ be a leaf of $T-X$. So, $d_T(y_i)=d_X(y_i)+1$ is even, which contradicts that $d_T(y_i)$ is odd.

\vspace{2mm} To show its necessity, suppose that $G$ has no odd spanning tree. Let $X=\{x_1, x_2, \ldots, x_s\}$, $Y=\{y_1, y_2, \ldots, y_t\}$. Let $F$ be a spanning forest of $G$ with exactly $t$ components $T_1, \ldots, T_t$, where $T_i$ is isomorphic to a star with $y_i$ as its center for each $i$. We say such a spanning forest of $G$ a {\it $Y$-star forest} of $G$.

\vspace{2mm}\noindent{\bf Claim 1.} Every $Y$-star forest has an even number of components with even size.
\begin{proof}
Let $V_1=\{y\in Y: d_F(y)$ is odd $\}$ and $V_2=Y\setminus V_1$. We show that $|V_2|$ is even.
Note that
$$s=\sum_{i=1}^t d_F(y_j)=\sum_{y\in V_1} d_F(y)+\sum_{y\in V_2} d_F(y),$$
$$t=|V_1|+|V_2|\ \text{and}\ s\equiv t\ (\hspace{-4mm}\mod 2)\ (\text{since}\ s+t\ \text{is even}).$$ If $s$ is odd, then $|V_1|$ is odd, and thus $|V_2|\equiv 0 (\mod 2)$; if $s$ is even, then $|V_1|$ is even, and thus $|V_2|\equiv 0 (\mod 2)$.
\end{proof}

\begin{center}
\scalebox{0.30}[0.30]{\includegraphics{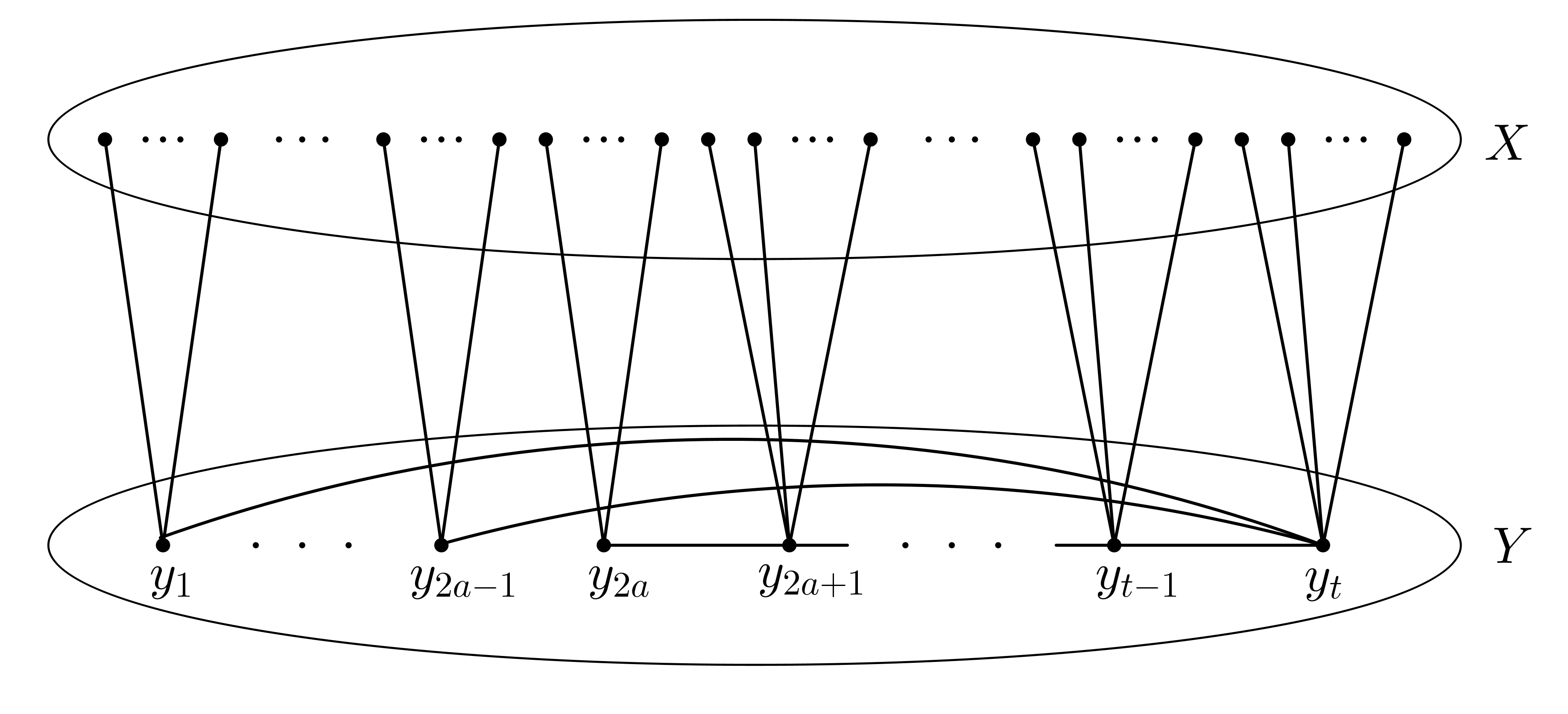}}
\centerline{{\bf Fig. 1}~~$T$.\hypertarget{Fig1}}
\end{center}

\noindent{\bf Claim 2.} If there exists a vertex $y\in Y$ such that
$d_X(y)\equiv 0~(\hspace{-2.5mm}\mod 2)$, or there exist two vertices $y_j$ and $y_k$ in $Y$ such that $N_X(y_j)\cap N_X(y_k)\neq\emptyset$, then $G$ has a $Y$-star forest with a component of even size.
\begin{proof} If there exists a vertex $y\in Y$ such that $d_X(y)\equiv 0~(\hspace{-2.5mm}\mod 2)$,
then $G$ has a $Y$-star forest $F$ with $d_F(y)=d_X(y)$. Clearly, $F$ is a $Y$-star forest, as we desired. In the second case, let $x_i$ be a common neighbor of $y_j$ and $y_k$. Let $F'$ be a $Y$-star-forest of $G-x_i$. Let $F=F'+x_iy_j$ if $d_{F'}(y_j)$ is odd, and otherwise $F=F'+x_iy_k$. Clearly, $F$ is a $Y$-star forest of $G$ as desired.
\end{proof}

\vspace{2mm} By Claims 1 and 2, let $F$ be a $Y$-star forest with an even number of components of even size. Without loss of generality, let $F=\cup_{i=1}^t T_i$, where $T_i$ is a component of even size for each $i\leq 2a$, and $T_i$ is a component of odd size for each $i\geq 2a+1$ for a positive integer $a$.
One can see that $$T=F+\{y_iy_{i+1}: 2a\leq i\leq t-1\}+\{y_jy_t: 1\leq j\leq 2a-1\}$$ is an odd spanning tree of $G$,  as illustrated in \hyperlink{Fig1}{Fig. 1}, a contradiction.
\end{proof}

A tree $T$ is said to a {\it double star} if there are exactly two vertices of degree at least two in $T$.

\begin{corollary}\label{corollary3.2}
Let $G$ be a connected graph with even order. If $diam(G)\geq 4$, then $\overline{G}$ has a spanning odd double star.
\end{corollary}
\begin{proof}
Let $v_0$ and $v_4$ be two vertices in $G$ with $d_G(v_0,v_4)=4$. Let $v_0v_1v_2v_3v_4$ be a shortest path joining $v_0$ and $v_4$ in $G$. Observe that $\overline{G}$ contains a split graph $G'$ as its spanning subgraph, where $G'=G'(X\cup Y, E)$ with $Y=\{v_0, v_4\}$ being a clique, $X=V(G)\setminus Y$ an independent set, and $E(G')=\{v_0v_4\}\cup E_0\cup E_4\cup E_{04}$, with $E_4=\{v_0'v_4: v_0'\in N_G(v_0)\}$, $E_0=\{v_4'v_0: v_4'\in N_G(v_4)\}$, and $E_{04}=\{vv_0, vv_4: v\in V(G)\setminus (N_G[v_0]\cup N_G[v_4])\}$. Since $v_2\in N_{\overline{G}}(v_0)\cap N_{\overline{G}}(v_4)$, $N_{\overline{G}}(v_0)\cap N_{\overline{G}}(v_4)\neq \emptyset$. By \hyperref[Theorem3.1]{Theorem 3.1}, $G'$ has an odd spanning tree, a spanning a double star. So, $G$ has an odd spanning tree as well.
\end{proof}

%
%
%
%
%

\section{\large Complements of triangle-free graphs}
In this section, we give a necessary and sufficient condition for a graph $G$ whose complement has an odd spanning tree.
Recall that a graph is claw-free if it contains no induced subgraph isomorphic to $K_{1,3}$.
Clearly, If $G$ is a triangle-free, then $\overline{G}$ is claw-free.
Since a graph has a spanning tree if and only if it is connected, it is natural to consider when $\overline{G}$ is connected for a triangle-free graph $G$. For $E'\subseteq E(G)$, $G[E']$ denotes the subgraph of $G$ with $E(G[E'])=E'$ and $V(G[E'])$ is the set of vertices incident with an edge of $E'$ in $G$.

\begin{lemma}\label{lemma4.1}
For a triangle-free graph $G$, $\overline{G}$ is disconnected if and only if $G$ is a complete bipartite graph.
\end{lemma}
\begin{proof}
If $G\cong K_{s,t}$ for some two positive integers $s$ and $t$, then $\overline{G}$ consists of two disjoint cliques $K_s$ and $K_t$, thus is disconnected.

If $\overline{G}$ is disconnected, then it has exactly two components. Otherwise, it has at least three components. Let $v_1, v_2, v_3$ be three vertices, which belong to distinct components of $\overline{G}$. Clearly, $v_1, v_2, v_3$ are pairwise adjacent in $G$, i.e. they induce a triangle in $G$, a contradiction. So, let $X$ and $Y$ be the vertex set of the two components of $\overline{G}$, respectively. Again, any vertex of $X$ is adjacent to all vertices of $Y$ in $G$. It follows $G[E_G[X,Y]]$ is a complete bipartite graph. Indeed, $G=G[E_G[X,Y]]$. Otherwise, $G$ must contain a triangle.
\end{proof}

For the seek of clarity, we divide triangle-free graphs into two families in terms of whether it is odd or not. We use $C_n$ to denote the cycle of order $n$. The complement of $C_4$ is denoted by $2K_2$.

\begin{theorem}\label{Theorem4.2}
Let $G$ be a triangle-free graph. If $G$ is odd, then $\overline{G}$ has an odd spanning tree if and only if $G\ncong 2K_2$ and $G$ is not a complete bipartite graph.
\end{theorem}

\begin{proof}
If $G$ is a complete bipartite graph, then $\overline{G}$ is disconnected, it has no spanning tree, and has no odd spanning tree either. If $G\cong 2K_2$, then $\overline{G}=C_4$. Clearly $C_4$ has no odd spanning tree.

\vspace{2mm} Assume $G$ is a triangle-free graph which is neither a complete bipartite graph nor $2K_2$. By \hyperref[lemma4.1]{Lemma 4.1}, $\overline{G}$ is connected. Fix a vertex $v\in V(G)$. For convenience, let $X=N_G(v)$ and $Y=V(G)\setminus N_G[v]$. Since $G$ is not a complete bipartite graph, there exist two vertices $x^*\in X$ and $y^*\in Y$ with $x^*y^*\notin E(G)$.

\vspace{2mm}\noindent {\bf Case 1.} There exists a vertex $y'\in Y\setminus \{y^*\}$ with $y^*y'\notin E(G)$.

\vspace{2mm} One can find an odd spanning tree $T_1$ of $\overline {G}$ with $$E(T_1)=\{x^*y^*, y^*y'\}\cup \{vy: y\in Y\setminus \{y'\}\}\cup \{x^*x: x\in X\setminus \{x^*\}\}$$ as shown in \hyperlink{Fig2}{Fig. 2(a)}, where $z$ is $y^*$.

\vspace{2mm}\noindent {\bf Case 2.} Any vertex $y\in Y\setminus \{y^*\}$ is adjacent to $y^*$ in $G$.

\vspace{2mm}\noindent {\bf Case 2.1.} $|X|=1$.

\vspace{2mm} We claim that $|Y|\geq 4$. 
If it is not, then $|Y|=2$, and thus the order of $G$ is 4. Since $G$ is odd, the assumption implies $G\cong 2K_2$, a contradiction. Note that $X$ is a clique in $\overline{G}$ with $|X|=d_G(v)$ being odd and $X\setminus \{x^*\} \subseteq N_{\overline{G}}(x^*)$. In addition, $x^*y^*\in E(\overline{G})$. Since $d_{\overline{G}}(x^*)$ is even, $x^*$ has a neighbor, say $y''\in Y$, other than $y^*$,
in $\overline{G}$. Take a vertex $y'\in Y\setminus \{y'', y^*\}$. By the assumption that $y^*y', y^*y''\in E(G)$, and $G$ is triangle-free, it follows that $y'y''\notin E(G)$, i.e. $y'y''\in E(\overline{G})$. Now one can find an odd spanning tree $T_1$ of $\overline{G}$ as shown in \hyperlink{Fig2}{Fig. 2(a)}, where $z$ is $y''$.

\begin{center}
\scalebox{0.30}[0.30]{\includegraphics{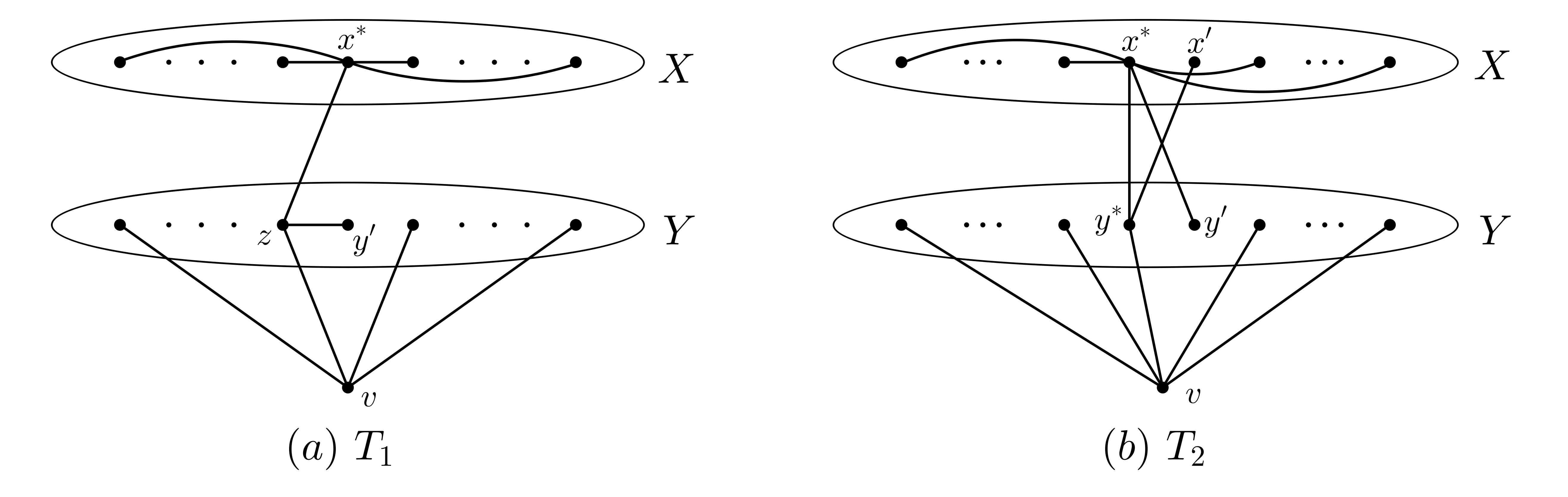}}
\centerline{{\bf Fig. 2}~~$(a)~T_1~\text{and}~(b)~T_2$.\hypertarget{Fig2}}
\end{center}

\vspace{2mm}\noindent {\bf Case 2.2.} $|X|\geq 2$.

\vspace{2mm} Since $|X|$ is odd, $|X|\geq 3$. Since $\overline{G}$ is connected and is even, there exists a vertex $y'\in Y\setminus \{y^*\}$ such that $x^*y'\in E(\overline{G})$. Let $x'\in X\setminus \{x^*\}$. Since $\overline{G}$ is claw-free, $x'$ must be adjacent to either $y^*$ or $y'$ in $\overline{G}$. Without loss of generality, let $x'y^*\in E(\overline{G})$. One can find an odd spanning tree $T_2$ of $\overline{G}$ with $$E(T_2)=\{x^*y^*, x^*y', x'y^*\}\cup \{x^*x: x\in X\setminus\{x^*,x'\}\}\cup \{vy: y\in Y\setminus \{y'\},$$as shown in \hyperlink{Fig2}{Fig. 2(b)}.
\end{proof}

Before proceeding, we introduce some additional notation. Let $C_5(k)$ be the graph by replacing one vertex $v$ of $C_5$ with an independent set of $k$ vertices, which are adjacent to the two neighbors of $v$ in $C_5$, as shown in \hyperlink{Fig3}{Fig. 3(a)}. In addition, $K_{2s, 2t}-e$ denotes the graph obtained from $K_{2s, 2t}$ by deleting an edge.

\begin{theorem}\label{Theorem4.3}
Let $G$ be a triangle-free graph of even order. If $G$ is not odd, then $\overline{G}$ has an odd spanning tree if and only if $G\ncong C_5(2)$ and $G\notin \{K_{2s, 2t}, K_{2s, 2t}-e\}$ for two positive integers $s$ and $t$.
\end{theorem}
\begin{proof}
To show its necessity, let $s$ and $t$ be any two positive integers.
If $G=K_{2s, 2t}$, then $\overline{G}$ is disconnected, and hence has no spanning (odd) tree.
If $G=K_{2s, 2t}-e$, then by \hyperref[Proposition2.5]{Proposition 2.5}, $\overline{G}$ has no odd spanning tree either. Next we show that $\overline{C_5(2)}$ has no an odd spanning tree.

To see this, first let $V(C_5(2))=\{x_1, x_2\}\cup \{y_i: 1\leq i\leq 4\}$ and $E(C_5(2))=\{x_iy_j: 1\leq i\leq 2, 3\leq j\leq 4\}\cup \{x_1y_2, x_2y_1, y_1y_2\}$.
Suppose $\overline{C_5(2)}$ has an odd spanning tree $T$. Clearly $x_1x_2\notin E(T)$, otherwise, one of $x_1$ or $x_2$ must have degree 2 in $T$. So, both $x_1$ and $x_2$ are leaves of $T$. It implies that $d_T(y_i)=3$. However, it forces that $T$ contains a 4-cycle $y_1y_3y_2y_4y_1$, a contradiction.

\begin{center}
\scalebox{0.30}[0.30]{\includegraphics{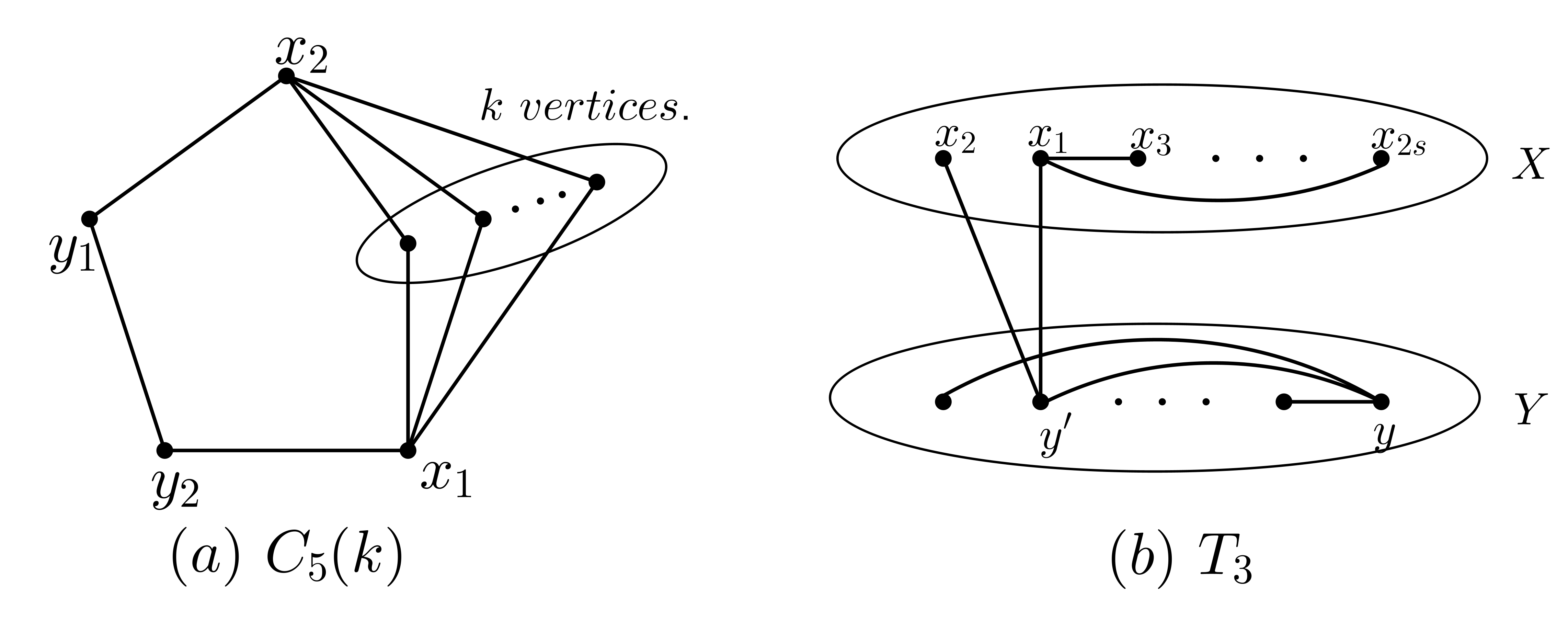}}
\centerline{{\bf Fig. 3}~~$(a)~C_5(k)~\text{and}~(b)\ T_3$.\hypertarget{Fig3}}
\end{center}

\vspace{2mm} Next, assume that $G$ is a triangle-free graph such that $\overline{G}$ has no odd spanning tree. We show that $G\cong C_5(2)$ or $G\in \{K_{2s, 2t}, K_{2s, 2t}-e\}$ for two positive integers $s$ and $t$.
Fix a vertex $y$ of $G$ with $d_G(y)$ being even.
For convenience, let $N_G(y)=X$, $Y=V(G)\setminus X$. Hence $y\in Y$. Since $G$ is triangle-free, $X$ is an independent set in $G$ and thus is a clique in $\overline{G}$. Since $d_G(y)$ is even, $X=\{x_1, x_2, \dots, x_{2s}\}$ and $Y=\{y_1, y_2, \dots, y_{2t}\}$. Trivially, $s>0$, otherwise $\overline{G}$ has a spanning odd star rooted at $y$, contradicting our assumption.

\vspace{2mm}\noindent{\bf Case 1.} There is a vertex $y'\in Y$ with $|N_G(y')\cap X|\leq 2s-2$.

\vspace{2mm} By the assumption, $y'\neq y$ and $|N_{\overline{G}}(y')\cap X|\geq 2$. Since $yy''\in E(\overline{G})$ for any $y''\in Y\setminus\{y\}$, one can find an odd spanning tree $T_3$ of $\overline{G}$ with
$$E(T_3)=\{yy'': y''\in Y\setminus \{y\}\}\cup \{y'x_1, y'x_2\}\cup \{x_1x_i: 3\leq i\leq 2s\},$$ as illustrated in \hyperlink{Fig3}{Fig. 3(b)}, a contradiction.

\vspace{2mm}\noindent{\bf Case 2.} $|N_G(y')\cap X|\geq 2s-1$ for any $y'\in Y$.

\begin{center}
\scalebox{0.30}[0.30]{\includegraphics{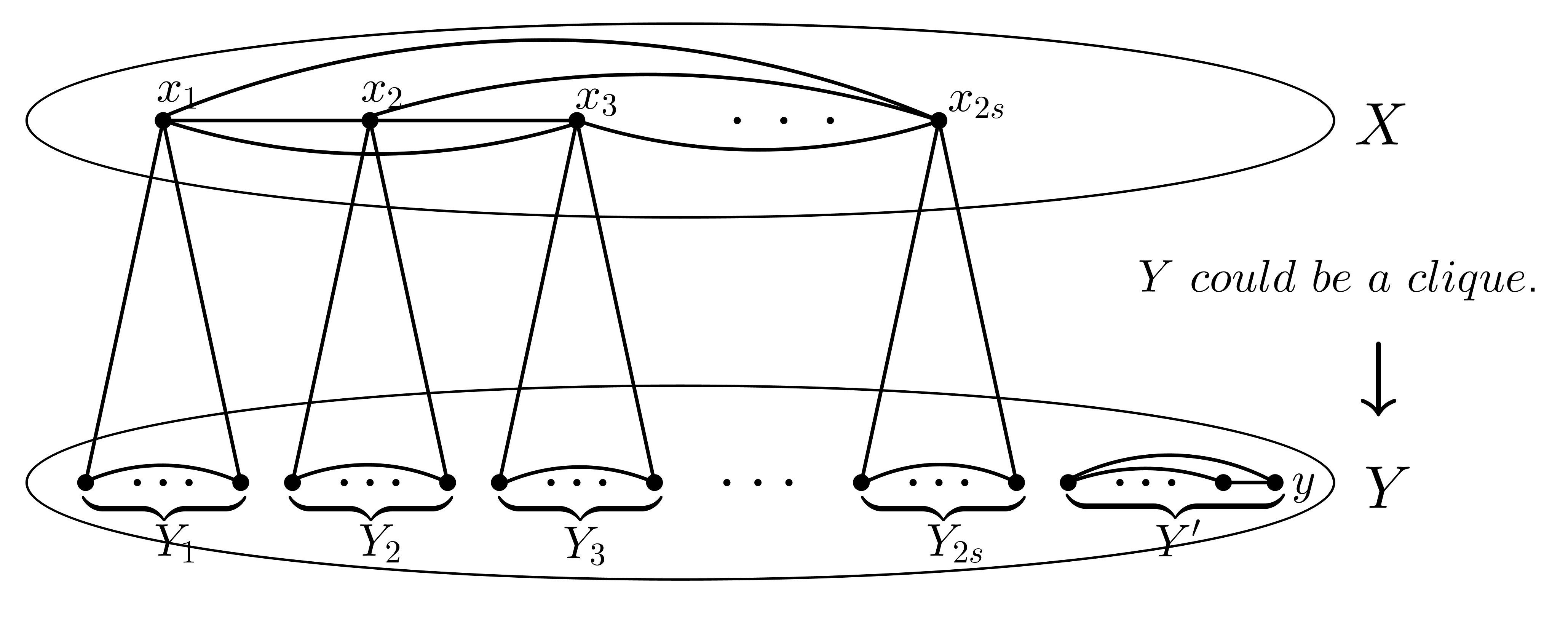}}
\centerline{{\bf Fig. 4}~~$\overline{G}$.\hypertarget{Fig4}}
\end{center}

\vspace{2mm} By the assumption, for any $y'\in Y$, \begin{equation}|N_{\overline{G}}(y')\cap X|\leq 1,\end{equation}.
If $|Y|=2$, then $G\in\{K_{2s,2}, K_{2s,2}-e\}$, as we desired. So, in what follows, assume that $2t=|Y|>2$.
So, it is naturally define $Y_i=N_{\overline{G}}(x_i)\cap Y$ for each $i\in\{1, \ldots, 2s\}$ and
$Y'=\{y': y'x\notin E(\overline{G})$ for all $x\in X\}$, as shown in \hyperlink{Fig4}{Fig. 4}. One can see that $Y'\cap Y_i=\emptyset$ and $Y_i\cap Y_j=\emptyset$ for any $i, j\in\{1,\ldots, 2s\}$. In particular, $Y'\neq \emptyset$ because of $y\in Y'$.
Indeed, $Y\setminus Y_i=N_G(x_i)$ for any $i$. Since $G$ is triangle-free, $Y\setminus Y_i$ is an independent set for each $i\in\{1, \ldots, 2t\}$. It follows that  $Y$ is independent if \begin{equation} \text{either}\ Y_i=\emptyset\ \text{for some}\ i\ \text{or}\ |\{i: 1\leq i\leq 2s\ \text{such that}\ Y_i\neq \emptyset\}|>2.\end{equation}

\vspace{2mm}\noindent{\bf Subcase 2.1.} $Y$ is an independent set of $G$.

\vspace{2mm}\noindent{\bf Subcase 2.1.1.} There exists an integer $i\in \{1, 2,\dots, 2s\}$ with $|Y_i|\geq 2$.

\vspace{2mm} Now one can find an odd spanning tree $T_4$ of $\overline{G}$ with
$$E(T_4)=\{x_iy_i', x_iy_i''\}\cup \{y_i'y': y'\in Y\setminus \{y_i', y_i''\}\cup \{x_ix: x\in X\setminus \{x_i\},$$ as illustrated in \hyperlink{Fig5}{Fig. 5}. This contradicts that $\overline{G}$ has no odd spanning tree.

\begin{center}
\scalebox{0.30}[0.30]{\includegraphics{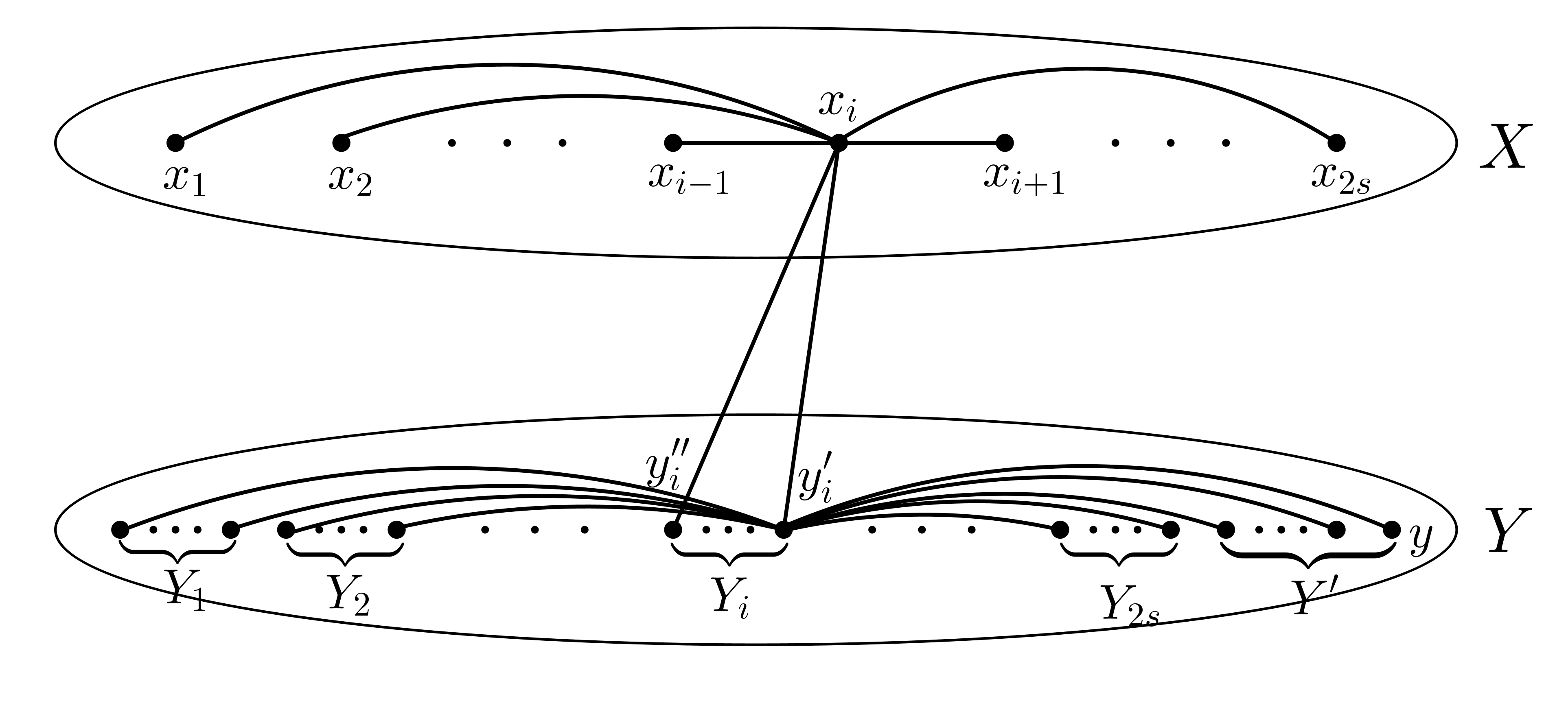}}
\centerline{{\bf Fig. 5}~~$T_4$.\hypertarget{Fig5}}
\end{center}

\vspace{2mm}\noindent{\bf Subcase 2.1.2.} $|Y_i|\leq 1$ for each $i\in \{1, 2,\dots, 2s\}$.

\vspace{2mm} It follows that $G\cong K_{2s,2t}\setminus M$, where $M$ is a matching of $K_{2s,2t}$. To show that $G\in\{K_{2s, 2t}, K_{2s, 2t}-e\}$, it suffices to show that
$|M|\leq 1$. Suppose $|M|\geq 2$. We show that $\overline{G}$ has an odd spanning tree.
Assume that $\{x_1y_1, x_2y_2\}\subseteq M$, without loss of generality.

\vspace{2mm} If $s\geq 2$, then $\overline{G}$ has an odd spanning tree $T_5$ with
$$E(T_5)=\{x_1y_1, x_2y_2\}\cup \{y_1y_i: 3\leq i\leq 2t\}\cup \{x_1x_2,x_2x_3\}\cup\{x_1x_i: 4\leq i\leq 2s\},$$ as shown in \hyperlink{Fig6}{Fig. 6}, a contradiction.

\begin{center}
\scalebox{0.30}[0.30]{\includegraphics{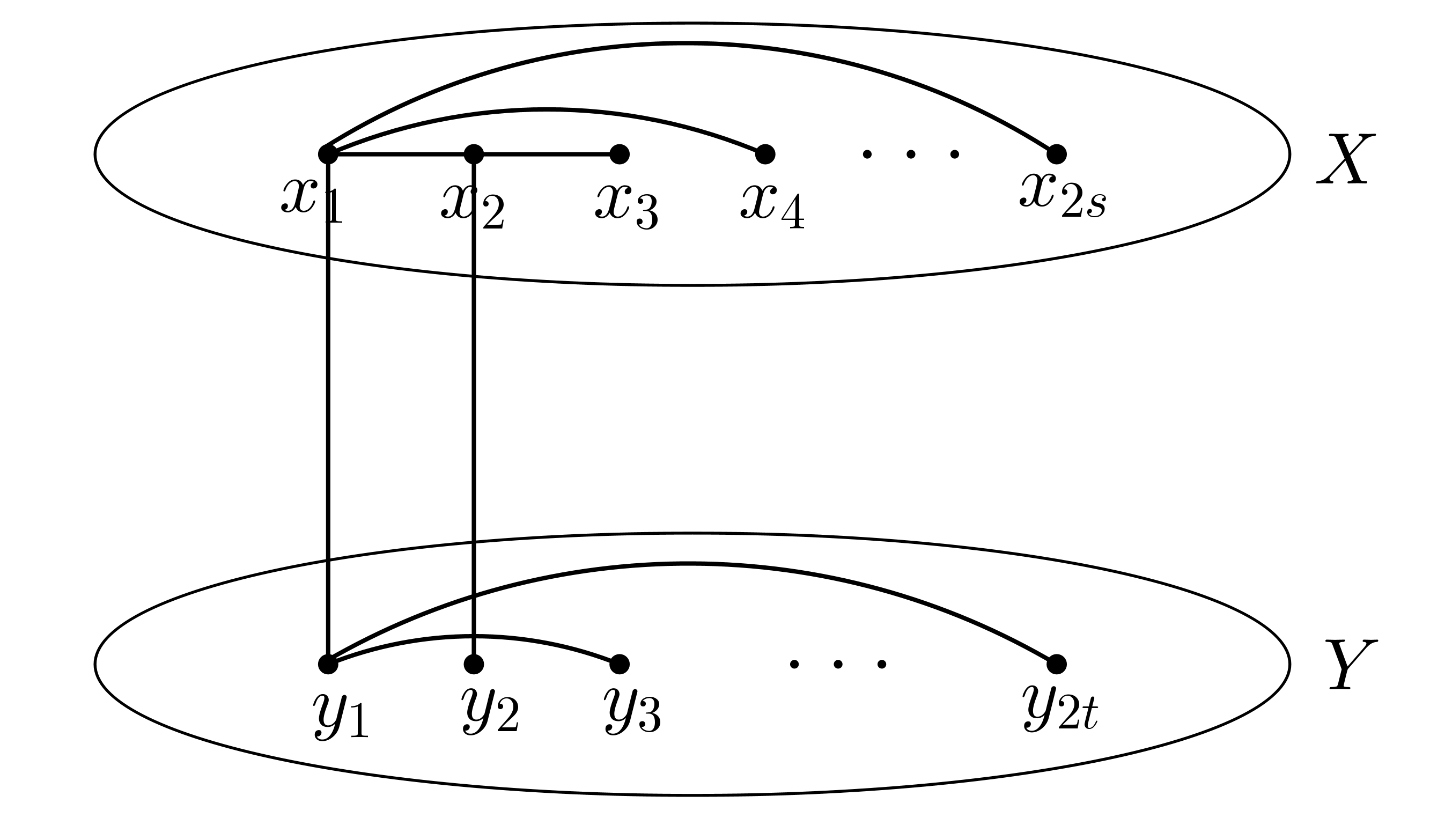}}
\centerline{{\bf Fig. 6}~~$T_5$.\hypertarget{Fig6}}
\end{center}

\vspace{2mm} If $s=1$, then $\overline{G}$ has an odd spanning tree $T_6$ with
$$E(T_6)=\{x_1y_1, x_2y_2\}\cup \{y_1y_2, y_1y_3\}\cup \{y_2y_i: 4\leq i\leq 2t\},$$ as illustrated in \hyperlink{Fig7}{Fig. 7}, a contradiction.

\begin{center}
\scalebox{0.30}[0.30]{\includegraphics{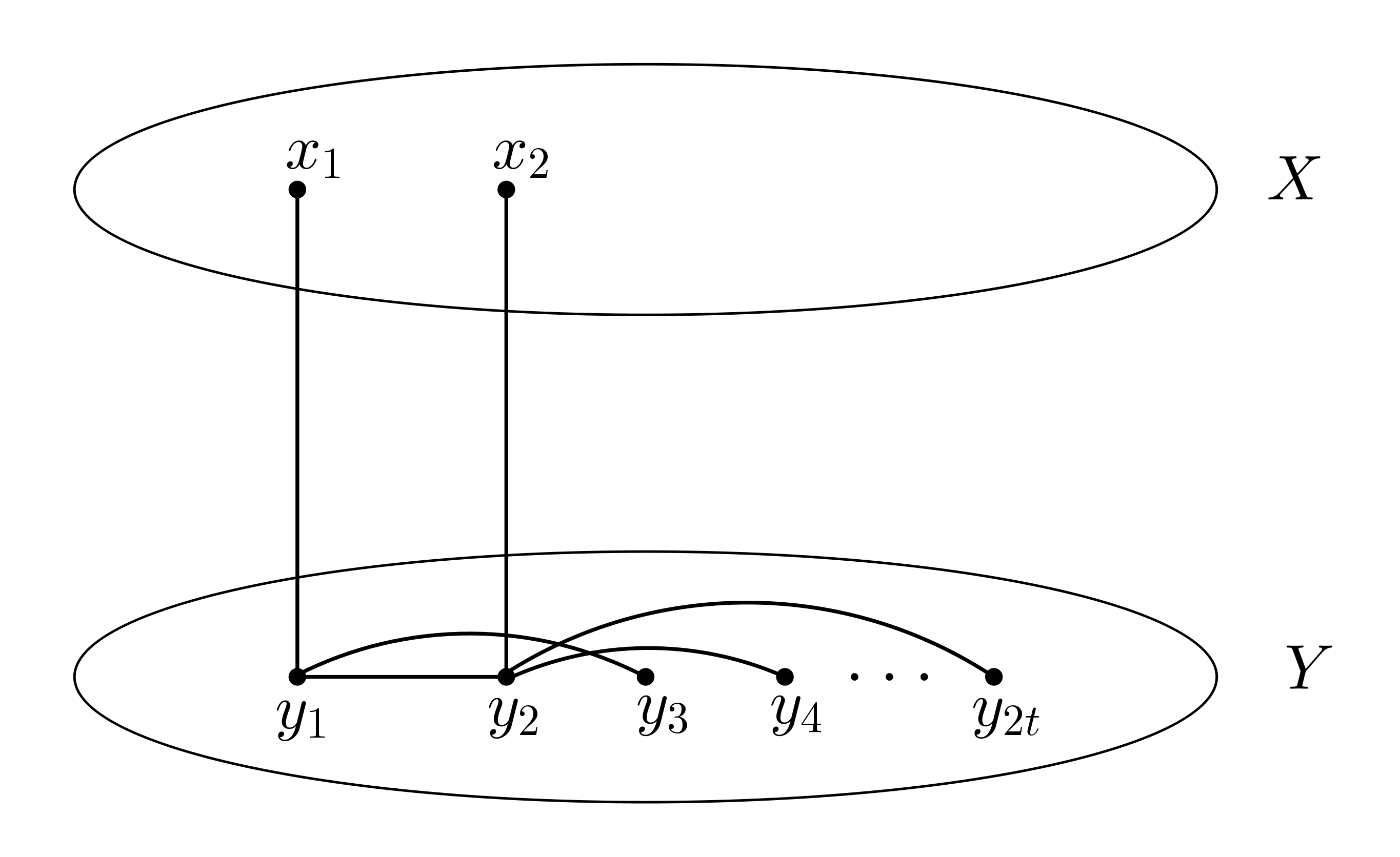}}
\centerline{{\bf Fig. 7}~~$T_6$.\hypertarget{Fig7}}
\end{center}

\vspace{2mm}\noindent{\bf Subcase 2.2.} $Y$ is not an independent set of $G$.

\vspace{2mm} Since $Y$ is not independent,  \begin{equation} Y_i\neq\emptyset\ \text{for any}\ i\ \text{and}\ |\{i: 1\leq i\leq 2s\ \text{such that}\ Y_i\neq \emptyset\}|\leq 2.\end{equation}
It means that $s=1$ and $Y_1\neq \emptyset$ and $Y_2\neq \emptyset$.
Recall that $Y'\neq \emptyset$. Since $Y$ is not independent, $E_G(Y_1, Y_2)\neq \emptyset$.

\vspace{2mm}\noindent {\bf Claim 1.} $E_{\overline{G}}(Y_1, Y_2)=\emptyset$
\begin{proof}
Suppose $E_{\overline{G}}(Y_1, Y_2)\neq\emptyset$, and let $y_1y_2\in E_{\overline{G}}(Y_1, Y_2)$, where $y_1\in Y_1$ and $y_2\in Y_2$.
First assume that $\max\{|Y_1|, |Y_2|\}\geq 2$. Without loss of generality, let $|Y_2|\geq 2$. One can see that $T_7$ is an odd spanning tree of $\overline{G}$, where
$$E(T_7)=\{x_1y_1, y_1y_2, x_2y_2, y_2y_3\}\cup \{yy': y'\in Y\setminus \{y, y_2, y_3\},$$as illustrated in \hyperlink{Fig8}{Fig. 8}, again a contradiction.

\begin{center}
\scalebox{0.30}[0.30]{\includegraphics{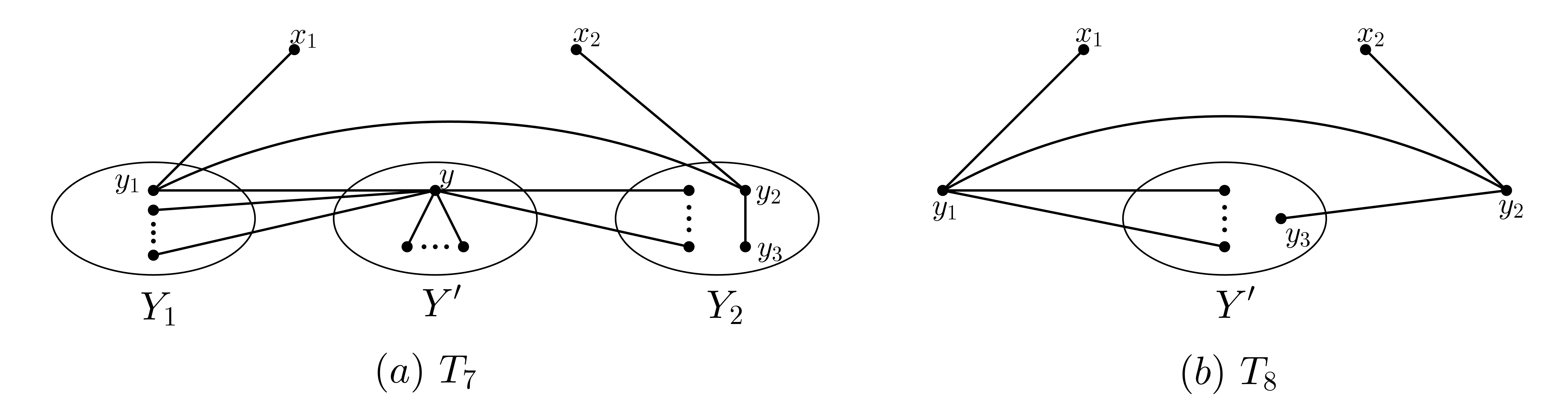}}
\centerline{{\bf Fig. 8}~~$(a)~T_7~\text{and}~(b)~T_8$.\hypertarget{Fig8}}
\end{center}

If $|Y_1|=|Y_2|=1$, then $T_8$ is an odd spanning tree of $\overline{G}$, where
$E(T_8)=\{x_1y_1, y_1y_2, x_2y_2, y_2y_3\}\cup \{y_1y': y'\in Y'\setminus \{y_3\}$, as illustrated in \hyperlink{Fig8}{Fig. 8}. This is a contradiction.
\end{proof}


\begin{center}
\scalebox{0.30}[0.30]{\includegraphics{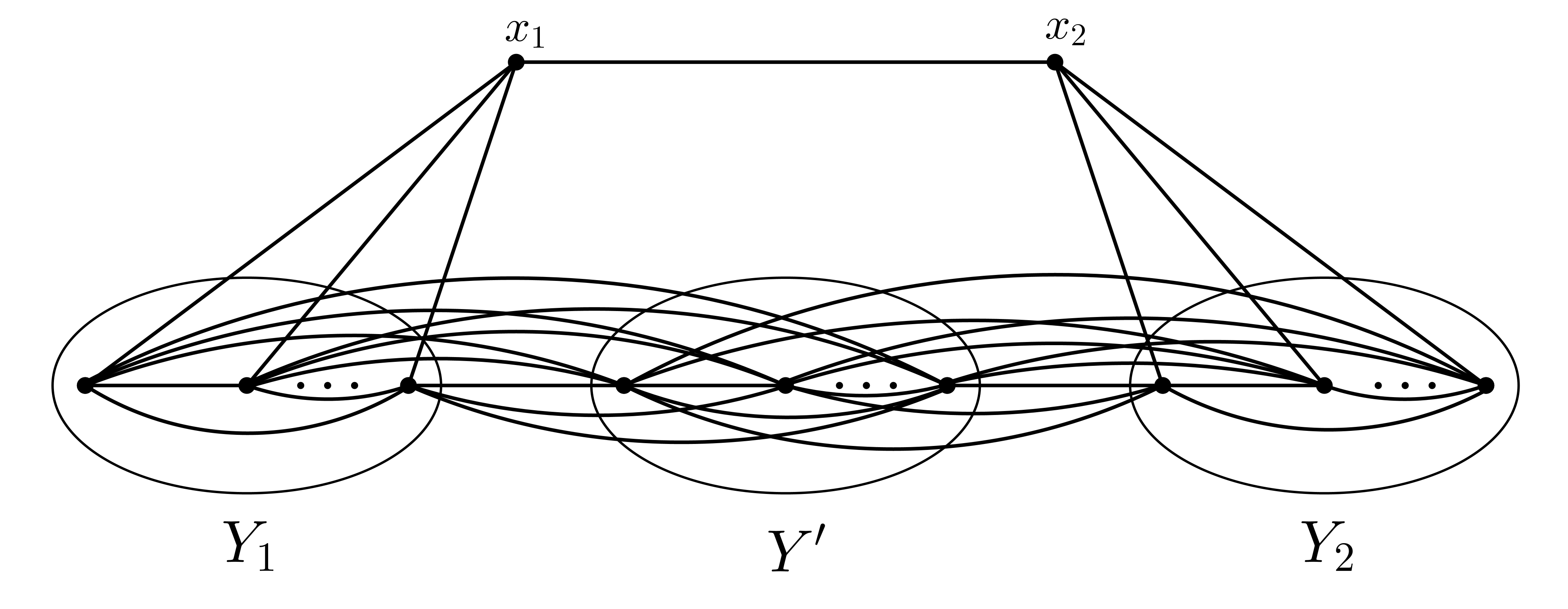}}
\centerline{{\bf Fig. 9}~~$\overline {G}$.\hypertarget{Fig9}}
\end{center}

By Claim 1, $\overline {G}$ is shown in  \hyperlink{Fig9}{Fig. 9}

\vspace{2mm}\noindent {\bf Claim 2.} $G\cong C_5(n-4)$
\begin{proof}

Suppose it is not. There are two possibilities.

\vspace{2mm}\noindent{\bf Subcase 2.2.1.} $\min\{|Y_1|,|Y_2|, |Y'|\}\geq 2$

\vspace{2mm} First assume that one of $|Y_1|$ and $|Y_2|$ is even.
Let $|Y_1|\equiv 0~(\hspace{-2.5mm}\mod 2)$, without loss of generality. One can find that $\overline{G}$ has an odd spanning tree $T_9$ with
\begin{align*}
  E(T_9)= & \{x_1x_2, x_1y_1, x_1y_1', y_1y\}\cup \{y_1y_1'': y_1''\in Y_1\setminus \{y_1, y_1'\}\} \\
   & \cup \{yy_2: y_2\in Y_2\}\cup\{yy'': y''\in Y'\setminus \{y, y'\}\}\cup \{y_1y'\},
\end{align*}
as shown in \hyperlink{Fig10}{Fig. 10}, contradicting our assumption.

\begin{center}
\scalebox{0.30}[0.30]{\includegraphics{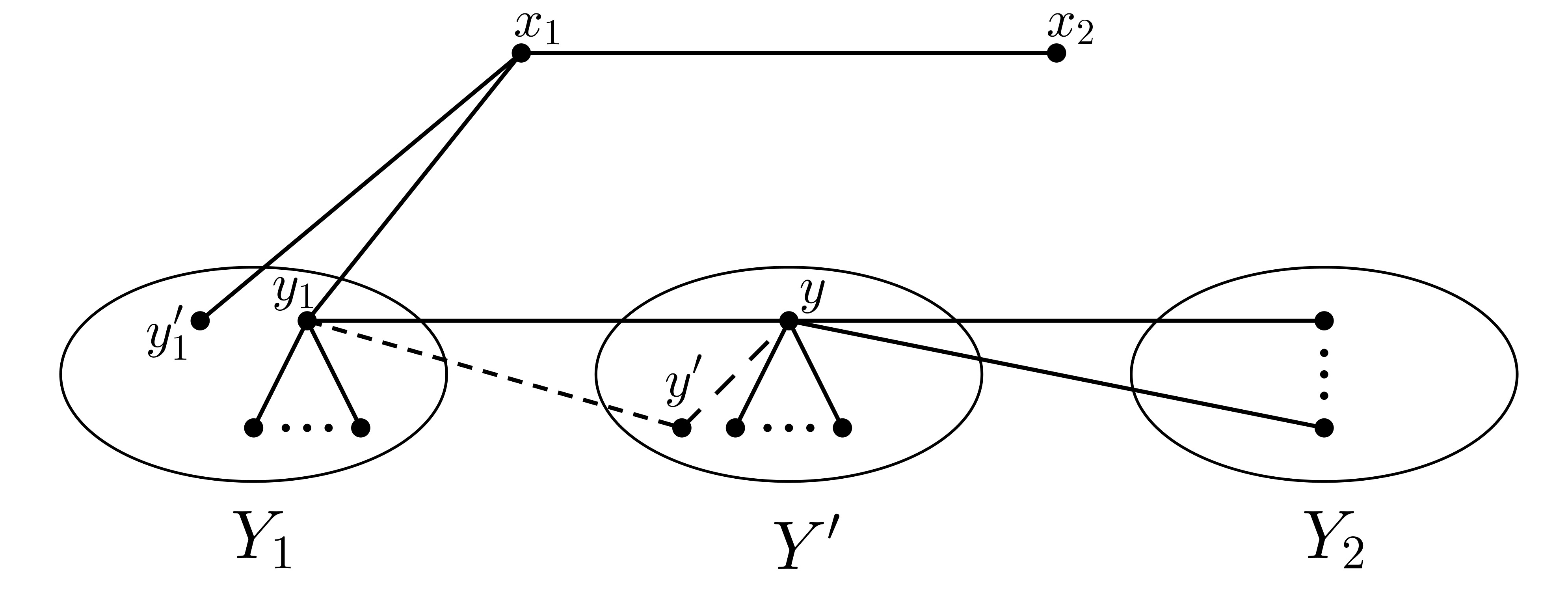}}
\centerline{{\bf Fig. 10}~~$T_9$\ \text{and}\ $T_{10}$.\hypertarget{Fig10}}
\end{center}

\vspace{2mm} If both $|Y_1|$ and $|Y_2|$ are odd,  $\overline{G}$ has an odd spanning tree $T_{10}$ with
$$E(T_{10})=E(T_9)-y_1y'+yy',$$ where $T_9$ is the tree constructed in the previous case, $T_{10}$ is illustrated in \hyperlink{Fig10}{Fig. 10}, a contradiction.

\vspace{2mm}\noindent{\bf Subcase 2.2.2.} Exactly one of $|Y_1|$, $|Y_2|$ and $|Y'|$ equals 1.

\vspace{2mm} First assume that exactly one of $Y_1$ and $Y_2$ has cardinality odd greater than 1. Without loss of generality, let $|Y_1|\equiv 1~(\hspace{-2.5mm}\mod 2)$ and $|Y_1|>1$, then $\overline{G}$ has an odd spanning tree $T_{10}$, as shown in \hyperlink{Fig10}{Fig. 10}, a contradiction.

\vspace{2mm} Now assume that one of $|Y_1|$ and $|Y_2|$ is 1 and the other is even. Without loss of generality, let $|Y_2|=1$. One can see that $T_9$ is an odd spanning tree of $\overline{G}$, as shown in \hyperlink{Fig10}{Fig. 10}, a contradiction.
\end{proof}

\begin{center}
\scalebox{0.30}[0.30]{\includegraphics{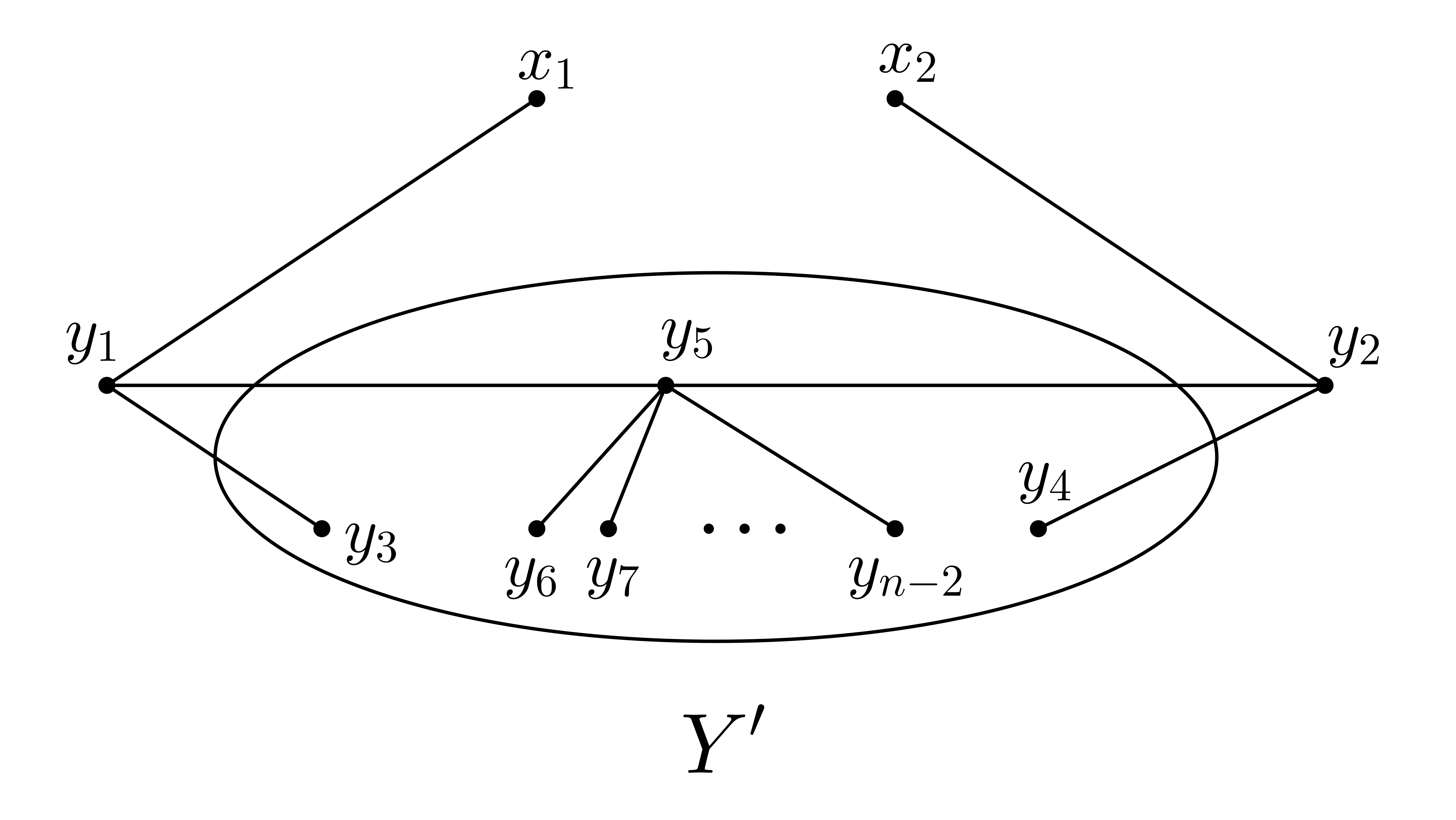}}
\centerline{{\bf Fig. 11}~~$T_{11}$.\hypertarget{Fig11}}
\end{center}

\vspace{2mm}\noindent {\bf Claim 3.} $n=6$
\begin{proof}

\vspace{2mm} By Claim 2, $G\cong C_5(n-4)$. Without loss of generality, let $Y_i=\{y_i\}$ for each $i\in\{1, 2\}$. Let $Y'=\{y_3, \ldots, y_{n-2}\}$. If $n\geq 8$, then $T_7$ is an odd spanning tree of $G$ with $$E(T_{11})=\{x_1y_1, y_1y_3, y_1y_5\}\cup\{y_5y_i: 6\leq i\leq n-2\}\cup \{y_2y_5, y_2y_4, x_2y_2\},$$ as shown in \hyperlink{Fig11}{Fig. 11}. This contradiction shows that $n=6$ i.e., $G\cong C_5(2)$.
\end{proof}

The proof is completed.
\end{proof}

\section{\large Concluding remarks}

In this paper, we present some sufficient conditions for a graph containing an odd spanning tree. For instance, we show that every graph of even order $n$ with $\delta(G)\geq \frac n 2 +1$ has an odd spanning tree. This is best possible, in sense that there are graph of order $n$ with $\frac n 2$ having no odd spanning tree. There are also some other related problems which are worth to be investigated.

(1) When does $\overline{G}$ have an odd spanning tree for a graph $G$ with $diam(G)\geq 3$?

(2) It is well known that the complement of a triangle-free graph is claw-free. We characterize all triangle-free graphs whose complements having an odd spanning tree. It is an interesting problem to characterize all connected claw-free graphs having an odd spanning tree. More specifically, when does a line graph have an odd spanning tree?

(3) We present a minimum degree condition $\delta(G)\geq \frac n 2+1$, which guarantees a graph $G$ of even order $n$ containing an odd spanning tree. Is these condition can be replaced with $d(u)+d(v)\geq n+2$ for any two nonadjacent vertices $u$ and $v$ in $G$, which forces $G$ having an odd spanning tree?

\vspace{3mm} Very recently, Liu and Wu \cite{Liu2025} proved that if $G$ is a graph of order even $n$ with $d(u)+d(v)\geq n+1$ for any two nonadjacent vertices $u$ and $v$ in $G$, then it contains an odd spanning tree.

\vspace{3mm}
\noindent{\bf Data Availability}

\vspace{2mm}No data was used for the research described in the article.

%


\begin{thebibliography}{}

\bibitem{Albertson1990}M. O. Albertson, D. M. Berman, J. P. Hutchinson, and C. Thomassen, Graphs with homeomorphically irreducible spanning trees, J. Graph Theory 14 (1990) 247-258.


\bibitem{Berman1997} D.M. Berman, A.J. Radcliffe, A.D. Scott, H. Wang, L. Wargo, All trees contain a large induced subgraph having all degrees 1 (mod $k$), Discrete Math. 175 (1997) 35-40.

\bibitem{Berman} D.M. Berman, H. Wang, L. Wargo, Odd induced subgraphs in graphs of maximum degree three, Australas. J. Combin. 15 (1997) 81-85.

\bibitem{Bondy} J.A. Bondy, U.S.R. Murty, Graph Theory, Graduate Texts in Math. 244, Springer, 2008.


\bibitem{Caro19941} Y. Caro, On induced subgraphs with odd degrees, Discrete Math. 132 (1994) 23-28.

\bibitem{Caro19942} Y. Caro, I. Krasikov, Y. Roditty, On induced subgraphs of trees with restricted degrees, Discrete Math. 125 (1994) 101-106.

\bibitem{Chen2012} G. Chen, H. Ren, S. Shan, Homeomorphically irreducible spanning trees in locally connected graphs, Combin. Probab. Comput. 21 (1-2) (2012) 107-111.

\bibitem{Chen2013} G. Chen, S. Shan, Homeomorphically irreducible spanning trees, J. Combin. Theory Ser. B 103 (2013) 409-414.

\bibitem{Diemunsch2015} J. Diemunsch, M. Furuya, M. Sharifzadeh, S. Tsuchiya, D. Wang, J. Wise and E. Yeager, A characterization of P5-free graphs with a homeomorphically irreducible spanning tree, Discrete Appl. Math. 185 (2015) 71-78.

\bibitem{Ferber2021} A. Ferber, M. Krivelevich, Every graph contians a linearly sized induced subgraph with all degrees odd, Adv. Math. 406 (2022), Paper No.108534, 11 pp.

\bibitem{Furuya2020} M. Furuya, S. Tsuchiya S. Large homeomorphically irreducible trees in path-free graphs, J. Graph Theory 93(2020) 372-394.

\bibitem{Gutin2016} G. Gutin, Note on perfect forests, J. Graph Theory 82 (2016) 233-235.

\bibitem{Hoffmann-Ostenhof2018} A. Hoffmann-Ostenhof, K. Noguchi, K. Ozeki, On homeomorphically irreducible spanning trees in cubic graphs, J. Graph Theory 89 (2018) 93-100.



\bibitem{Hou2018} X. Hou, L. Yu, J. Li, B. Liu, Odd induced subgraphs in graphs with treewidth at most two, Graphs Combin. 34 (2018) 535-544.
    
\bibitem{Liu2025} Y. Liu, B. Wu, Graphs with odd spanning trees, in preparation. 

\bibitem{Ito2020} T. Ito, Y. Nakamura, S. Tsuchiya, Homeomorphically irreducible spanning trees in small graphs, Information Science and Applied Mathematics 28 (2020) 1-16.

\bibitem{Ito2022} T. Ito, S. Tsuchiya, Degree sum conditions for the existence of homeomorphically irreducible spanning trees, J. Graph Theory 99 (2022) 162-170.



\bibitem{Lovasz1979} L. Lov\'{a}sz, Combinatorial Problems and Exercises (North-Holland, Amsterdam), 1979.

\bibitem{Nash-Williams1961} C.St.J.A. Nash-Williams, Edge-disjoint spanning trees of finite graphs, J. London Math. Soc. 36 (1961) 445-450.

\bibitem{Radcliff1995} A.J. Radcliff, A.D. Scott, Every tree contains a large induced subgraph with all degrees odd, Discrete Math. 140 (1995) 275-279.

\bibitem{Rao2022} M. Rao, J. Hou, Q. Zeng, Odd induced subgraphs in planar graphs with large girth, Graphs Combin. 38(4) (2022), Paper No. 105, 12 pp.

\bibitem{Scott1992} A.D. Scott, Large induced subgraphs with all degrees odd, Combin. Probab. Comput. 1 (1992) 335-349.

\bibitem{Scott2001} A.D. Scott, On induced subgraphs with all degrees odd, Graphs Combin. 17 (2001) 539-553.

\bibitem{Shan2023} S. Shan, S. Tsuchiya, Characterization of graphs of diameter 2 containing a homeomorphically irreducible spanning tree, J. Graph Theory 104 (2023) 886-903.

\bibitem{Tutte1961} W.T. Tutte, On the problem of decomposing a graph into $n$ connected factors, J. London Math. Soc. 36 (1961) 221-230.

\bibitem{Wang2023} T. Wang, B. Wu, Induced subgraphs of a tree with constraint degree, Appl. Math. Comput. 440 (2023) 127657.

\bibitem{WangWu2024} T. Wang, B. Wu, Maximum odd induced subgraph of a graph concerning its chromatic number, J. Graph Theory 107 (2024) 578-596.

\bibitem{Zhai2019} S. Zhai, E. Wei, J. He, D. Ye, Homeomorphically irreducible spanning trees in hexangulations of surfaces, Discrete Math. 342 (2019) 2893-2899.
\end{thebibliography}
\end{document}